\documentclass[11pt]{article}
\usepackage{amsmath}
\usepackage{amsthm}
\usepackage{amsfonts}
\usepackage{amssymb}
\usepackage{latexsym}
\usepackage{bm}

\usepackage[margin=2.9cm]{geometry}
\usepackage{tikz}

\usepackage{subcaption}

\usepackage{pgfplots}
\usepackage{mathrsfs}
\usetikzlibrary{arrows}

\usepackage[colorlinks,
linkcolor=blue,
anchorcolor=blue,
citecolor=blue
]{hyperref}

\def\pf{\noindent {\it Proof.} }
\newtheorem{thm}{Theorem}
\newtheorem{lem}[thm]{Lemma}
\newtheorem{prop}[thm]{Proposition}

\newtheorem{rem}[thm]{Remark}

\newcommand{\Gen}{\mathop{\mathbf{Gen}}\nolimits}

\begin{document}

\begin{center}
	{\large \bf Stable  multivariate  Narayana polynomials and labeled  plane trees}
\end{center}

\begin{center}
	Harold R.L. Yang$^{a}$ and  Philip B. Zhang$^{b}$
	\\[6pt]

	$^{a}$College of Science\\
	Tianjin University of Technology and Education \\
	Tianjin 300222,
	R. R. China\\[6pt]

	$^{b}$College of Mathematical Science \\
	Tianjin Normal University, Tianjin  300387, P. R. China\\[6pt]

	Email:  $^{a}${\tt yangruilong@tute.edu.cn},
	$^{b}${\tt zhang@tjnu.edu.cn}
\end{center}

\noindent\textbf{Abstract.}
In this paper, we introduce stable multivariate generalizations of Narayana polynomials of type A and type B.
We give an insertion algorithm for labeled plane trees and introduce the notion of improper edges.
Our polynomials are multivariate generating polynomials of labeled plane trees and can be generated by a grammatical labeling based on a context-free grammar. Our proof of real stability uses a characterization of stable-preserving linear operators due to Borcea and  Br\"and\'en.
In particular, we get an alternative multivariate stable refinement of the second-order Eulerian polynomials, which is different from the one given by Haglund and Visontai.

\noindent {\bf AMS Classification 2020:} 05A15, 26C10


\section{Introduction}\label{intro-sec}
Univariate polynomials with only real roots often occur in combinatorics and can lead to interesting conclusions.
There are several real-rooted univariate polynomials in combinatorics, including Eulerian polynomials, Narayana polynomials, matching polynomials of graphs, and independence polynomials of claw-free polynomials.  We refer to the following surveys~\cite{Braenden2015Unimodality, Brenti1989Unimodal, Brenti1994Log, Stanley1989Log}.

Due to the work of  Borcea and Br\"and\'en~\cite{Borcea2009Lee, Borcea2009Leea, Borcea2009Polya}, a successful multivariate generalization of real-rooted polynomials, called \emph{stable polynomials}, have been developed.
Now let us recall the notion of real stability, which generalizes the
notion of real-rootedness from univariate real polynomials to multivariate real polynomials.
Given a positive integer $n$, let $\bm{x}$ be an  $n$-tuple $(x_1, \dots, x_n)$.
Let $\mathbb{H}_+ = \{z \in \mathbb{C} : \mbox{Im}(z)>0 \}$ denote the open upper complex
half-plane.
A polynomial $f \in \mathbb{R}[\bm{x}]$ is said to be (real) \emph{stable} if $f(\bm{x}) \neq 0$  for any
$\bm{x} \in \mathbb{H}_+^n$ or $f$ is identically zero.
Note that a univariate polynomial $f(x) \in \mathbb{R}[x]$ is stable if and only if it has only real roots.
Multivariate stable polynomials have been
studied in many different areas, including control theory, statistical mechanics, partial differential equations, and functional analysis.
See the surveys~\cite{Braenden2015Unimodality, Pemantle2012Hyperbolicity, Wagner2011Multivariate} for further applications of stable polynomials.
It is remarkable that an extension of stable polynomials, namely \emph{Lorentzian polynomials}, was studied by Br\"and\'en\ and Huh~\cite{Braenden2020Lorentzian}.

A natural question to ask is to find a stable multivariate generalization of real-rooted univariate polynomials in combinatorics.
Several multivariate polynomials have been found in combinatorics.
The multivariate Eulerian polynomials and more general Eulerian-like
polynomials were introduced by Br\"and\'en, Haglund, Visontai, and Wagner in \cite{Braenden2011Proof}
and used to prove a conjecture of Haglund, Ono, and Wagner \cite{Braenden2011Proof}.
Subsequently, multivariate Eulerian polynomials over Stirling permutations~\cite{Chen2020Context, Haglund2012Stable},
colored permutations~\cite{Braenden2016Multivariate, Visontai2013Stable}, segmented permuations~\cite{Zhang2019Multivariate}, and related combinatorial models have been studied.

In this paper, we focus on finding stable multivariate generalizations of two kinds of Narayana polynomials.
The Narayana polynomial of type A,  defined as
$
	N_n^A(x)=\sum_{k=1}^n  \frac{1}{n}{n\choose k}{n\choose k-1} x^k,
$
is  a $q$-analogue of  the famous Catalan number, and the coefficent of $x^k$ in $N_n^A(x)$, usually called \emph{Narayana number},  counts Dyck paths of semilength $n$ with $k$ peaks, or unlabeled  plane trees with $n$ edges and $k$ leaves, see \cite[A001263]{SloaneOnline}.
Remarkablely, Narayana polynomial of type A
is also the $h$-polynomial of the simplicial complex dual to an associahedron of type $A_n$, see \cite{Fomin2007Root}.
The Narayana polynomial of type B,  defined as
$
	N_n^B(x)=\sum_{k=0}^n {n\choose k}^2x^k,
$
is the $h$-polynomial of the simplicial complex dual to an associahedron of type $B_n$ (a cyclohedron), see \cite[A008459]{SloaneOnline}.
It is known that both $N_n^A(x)$ and $N_n^B(x)$ are real-rooted for any integer $n\geq 1$, see  Br\"and\'en~\cite{Braenden2006linear}, and Liu and Wang~\cite{Liu2007unified}.

Our multivariate polynomials are based on labeled plane trees.
A labeled plane tree with $n$ nodes refers to a rooted plane tree in which each node is assigned a unique label from the set $[n]=\{1,2, \ldots, n\}$.
Denote by $\mathcal{T}_{n}$ the set of labeled plane trees with $n$ nodes and by $\mathcal{T}^{*}_{n}$ the set of labeled plane trees with  $n+2$ nodes in which the node $\mathbf{1}$ is the leftmost leaf of the node $\mathbf{2}$.
Now we define
$$F_n(\mathbf{x}, \mathbf{y}, s, t) = \sum_{T\in \mathcal{T}_{n+1}} \mathbf{wt}(T)$$
and
$$F^{*}_n(\mathbf{x}, \mathbf{y}, s, t) = \sum_{T\in \mathcal{T}^*_{n}} \mathbf{wt}(T).$$
Here the details of the definition of $\mathbf{wt}(T)$ will be given in Section~\ref{section-4}.
Note that labeled plane trees are in bijections with quasi-Stirling permutations, which have been extensively studied by Elizalde~\cite{Elizalde2021Descents}, and Yan and Zhu~\cite{Yan2022Quasi}.

A Cayley tree refers to a tree where the children of each node are unordered.
It is well-known there are $n^{n-2}$ Cayley trees on $[n]$. To give a combinatorial proof of this enumerative result, Shor~\cite{Shor1995new} introduced the notion of improper edges, and provided a method to construct Cayley trees on $[n+1]$ from Cayley trees on $[n]$.
Inspired by Shor's work, we give an insertion algorithm for labeled plane trees and introduce the notion of improper edges.

Our construction of $F_n(\mathbf{x}, \mathbf{y}, s, t)$ and $F^{*}_n(\mathbf{x}, \mathbf{y}, s, t)$ is via a context-free grammar as well as a grammatical labeling of labeled plane trees.
Chen~\cite{Chen1993Context} initiated the study of applications of context-free grammars to combinatorics. Chen and Fu~\cite{Chen2017Context} introduced the notations of grammatical labelings to generate combinatorial structures, such as permutations and increasing trees, via a recursive process.
Further study on grammatical labelings related to trees are given in~\cite{Chen2023Gessel, Chen2020Context, Chen2021context}.

The main result of this paper is as follows.

\begin{thm}\label{thm-main}
	For any positive integer $n$, we have
	\begin{itemize}
		\item[(i)] $F_n(\mathbf{x}, \mathbf{y}, s, t)$ and $F^*_n(\mathbf{x}, \mathbf{y}, s, t)$
		      are multivariate generalizations of  $N_n^A(x)$ and $N_n^B(x)$, respectively. Namely,
		      \begin{align}
			      F_n(x,x,\ldots,x,1,1,\ldots,1, t, t)=(n+1)!t^nN_n^A(x),
		      \end{align}
		      and
		      \begin{align}
			      F^*_n(x,x,\ldots,x,1,1,\ldots,1, t, t)=n!t^{n+1}N_n^B(x).
		      \end{align}
		\item[(ii)]  $F_n(\mathbf{x}, \mathbf{y}, s, t)$ and $F^*_n(\mathbf{x}, \mathbf{y}, s, t)$
		      are real stable over variables $\mathbf{x}$ and  $\mathbf{y}$ for any positive real numbers $s$ and $t$.
	\end{itemize}
\end{thm}

The rest of this paper is organized as follows.
In Section~\ref{section-2}, we introduce an insertion algorithm for labeled plane trees.
Section~\ref{section-3} is devoted to a context-free grammar as well as a grammatical labeling which can be used to generate labeled plane trees. Our grammar can be transformed to a bivariate grammar of Ma, Ma, and Yeh~\cite{Ma2019gamma}  during their study of Narayana polynomials.
As applications of our grammar, we give grammatical proofs of convolution identities and generating functions of $N_n^A(x)$ and $N_n^B(x)$.
In Section~\ref{section-4}, we define a refined grammatical labeling for labeled plane trees to prove Theorem~\ref{thm-main} by Borcea and Br\"and\'en's characterization of stability-preserving linear operators.
We also present two applications of Theorem~\ref{thm-main}. In particular, we give an alternative multivariate stable refinement of second-order Eulerian polynomials, distinct from that provided by Haglund and Visontai~\cite{Haglund2012Stable}.

\section{An insertion algorithm for labeled plane trees}~\label{section-2}

For each edge of the tree, we call the node closer to the root as the \emph{parent} node of the edge, and the node farther from the root as the \emph{child} node of the edge.
Nodes with the same parent are called \emph{siblings} and the siblings to the left of a node $v$ are called \emph{elder} siblings of $v$.
A node is called a \emph{leaf} if it has no children and an \emph{interior node} otherwise.
Following Chen~\cite{Chen2006Old}, a node is called
\emph{old} if it is the leftmost child of its parent; otherwise, it is \emph{young}.
For our convenience, we call the node labeled $i$ as node $\mathbf{i}$.

In this section, we provide an insertion algorithm for labeled plane trees by adding a new node $\mathbf{n+1}$  from a tree in $\mathcal{T}_{n}$ via one of the following manners:


\begin{itemize}
	\item[N1:]  Choose a node $\mathbf{i}$, which can be a leaf or an interior node, and add $\mathbf{n+1}$ to be the old leaf of~$\mathbf{i}$. In the new tree, the node $\mathbf{n+1}$ is an old leaf.

	      
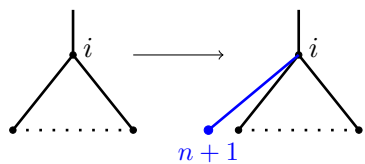
\begin{center}
		      \begin{tikzpicture}[scale=0.2]
			      \draw [line width=1pt] (-15,0)-- (-15,3);
			      \draw [line width=1pt] (-15,0.)-- (-15-4,-5);
			      \draw (1-15,0.5) node {$i$};
			      \draw [line width=1pt] (-15,0)-- (-15+4,-5.);
			      \draw [line width=1pt, loosely dotted] (-4-15,-5.)-- (4-15,-5.);
			      \draw [fill=black] (-15,0.) circle (6pt);
			      \draw [fill=black] (-4-15,-5.) circle (6pt);
			      \draw [fill=black] (4-15,-5.) circle (6pt);

			      \draw[->] (-11,0) -- (-5,0);

			      \draw [line width=1pt] (0,0.)-- (0.,3);
			      \draw [line width=1pt] (0,0.)-- (-4,-5);
			      \draw [line width=1pt] (0,0.)-- (4,-5.);
			      \draw [line width=1pt, loosely dotted] (-4.,-5.)-- (4.,-5.);
			      \draw [fill=black] (0.,0.) circle (6pt);

			      \draw [fill=black] (-4.,-5.) circle (6pt);
			      \draw [fill=black] (4.,-5.) circle (6pt);
			      \draw (1,0.5) node {$i$};
			      \draw [color=blue,fill=blue] (-6.,-5.) circle (8pt);
			      \draw [line width=1pt, blue] (0.,0.)-- (-6.,-5.);
			      \draw[color=blue] (-6,-6.5) node {\small{$n+1$}};
		      \end{tikzpicture}
	      \end{center}

	\item[N2:]

	      Choose a node $\mathbf{i}$, relabel the node $\mathbf{i}$ by $\mathbf{n+1}$, and add a new node $\mathbf{i}$ as the old leaf of $\mathbf{n+1}$.
	      In the new tree, the node $\mathbf{n+1}$ is an interior node whose old child is a leaf.
	      
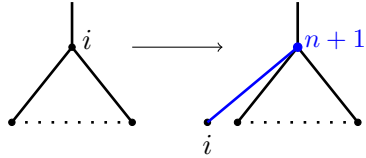
\begin{center}
		      \begin{tikzpicture}[scale=0.2]
			      \draw [line width=1pt] (-15,0)-- (-15,3);
			      \draw [line width=1pt] (-15,0.)-- (-15-4,-5);
			      \draw (1-15,0.5) node {$i$};
			      \draw [line width=1pt] (-15,0)-- (-15+4,-5.);
			      \draw [line width=1pt, loosely dotted] (-4-15,-5.)-- (4-15,-5.);
			      \draw [fill=black] (-15,0.) circle (6pt);
			      \draw [fill=black] (-4-15,-5.) circle (6pt);
			      \draw [fill=black] (4-15,-5.) circle (6pt);

			      \draw[->] (-11,0) -- (-5,0);

			      \draw [line width=1pt] (0,0.)-- (0.,3);
			      \draw [line width=1pt] (0,0.)-- (-4,-5);
			      \draw [line width=1pt] (0,0.)-- (4,-5.);
			      \draw [line width=1pt, loosely dotted] (-4.,-5.)-- (4.,-5.);
			      \draw [color=blue,fill=blue] (0.,0.) circle (8pt);

			      \draw [fill=black] (-4.,-5.) circle (6pt);
			      \draw [fill=black] (4.,-5.) circle (6pt);
			      \draw [color=blue]  (2.5, 0.5) node {\small{$n+1$}};
			      \draw [color=black,fill=black] (-6.,-5.) circle (6pt);
			      \draw [line width=1pt, blue] (0.,0.)-- (-6.,-5.);
			      \draw [color=black] (-6,-6.5) node {$i$};
		      \end{tikzpicture}
	      \end{center}

	\item[E1:]
	      Choose an edge $(\mathbf{i},\mathbf{j})$, and add $\mathbf{n+1}$ to be a leaf as  the younger brother of $\mathbf{j}$. In the new tree, the node $\mathbf{n+1}$ is a young leaf.

	      
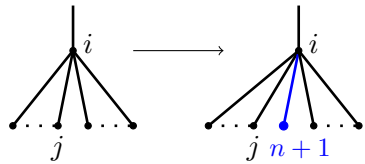
\begin{center}
		      \begin{tikzpicture}[scale=0.2]
			      \draw [line width=1pt] (-15,0)-- (-15,3);
			      \draw [line width=1pt] (-15,0.)-- (-15-4,-5);
			      \draw (1-15,0.5) node {$i$};
			      \draw [line width=1pt] (-15,0)-- (-15+4,-5.);
			      \draw [line width=1pt] (-15,0)-- (-15+1,-5.);
			      \draw [line width=1pt, loosely dotted] (-4-15,-5.)-- (-1-15,-5.);
			      \draw [line width=1pt, loosely dotted] (1-15,-5.)-- (4-15,-5.);
			      \draw [fill=black] (-15,0.) circle (6pt);
			      \draw [fill=black] (-4-15,-5.) circle (6pt);
			      \draw [fill=black] (4-15,-5.) circle (6pt);
			      \draw [fill=black] (-1-15,-5.) circle (6pt);
			      \draw [fill=black]  (1-15,-5) circle (6pt);
			      \draw [line width=1pt] (-15,0.)-- (-15-1,-5);
			      \draw (-1-15,-6.5) node {$j$};

			      \draw[->] (-11,0) -- (-5,0);

			      \draw [line width=1pt] (0,0)-- (0,3);
			      \draw [line width=1pt] (0,0)-- (-3,-5);
			      \draw [line width=1pt, blue] (0,0)-- (-1,-5);
			      \draw [line width=1pt] (0,0)-- (1,-5);
			      \draw [line width=1pt] (0,0)-- (4,-5);
			      \draw [line width=1pt, loosely dotted] (-6,-5)-- (-3,-5);
			      \draw [line width=1pt, loosely dotted] (1,-5)-- (4,-5);
			      \draw [fill=black] (0.,0.) circle (6pt);

			      \draw [fill=black] (-3,-5) circle (6pt);
			      \draw [fill=blue, blue] (-1,-5) circle (8pt);
			      \draw [fill=black] (1,-5) circle (6pt);
			      \draw [fill=black] (4,-5) circle (6pt);
			      \draw (1,0.5) node {$i$};
			      \draw [fill=black]  (-6,-5) circle (6pt);
			      \draw [line width=1pt] (0,0)-- (-6,-5);
			      \draw  (-3,-6.5) node {$j$};
			      \draw [color=blue] (0.1,-6.5) node {\small{$n+1$}};
		      \end{tikzpicture}
	      \end{center}

	\item[E2:]
	      Choose an edge $(\mathbf{i},\mathbf{j})$. We relabel $\mathbf{i}$ by $\mathbf{n+1}$ and make $\mathbf{i}$ be the old child of $\mathbf{n+1}$. Meanwhile,  assign the elder siblings of  $\mathbf{j}$ as well as itself to be the children of $\mathbf{i}$, and assign the remaining children of $\mathbf{i}$ to be the younger siblings of $\mathbf{i}$. In the new tree, the node $\mathbf{n+1}$ is an interior node whose first child is interior.
	      Note that a similar operation appeared in Shor's paper~\cite{Shor1995new}.

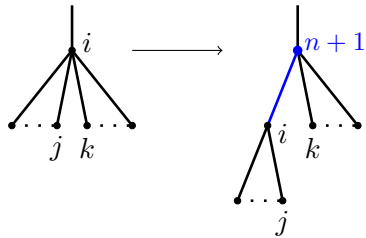
	      \begin{center}
		      \begin{tikzpicture}[scale=0.2]
			      \draw [line width=1pt] (-15,0)-- (-15,3);
			      \draw [line width=1pt] (-15,0.)-- (-15-4,-5);
			      \draw (1-15,0.5) node {$i$};
			      \draw  (1-15,-6.5) node {$k$};
			      \draw [line width=1pt] (-15,0)-- (-15+4,-5.);
			      \draw [line width=1pt] (-15,0)-- (-15+1,-5.);
			      \draw [line width=1pt, loosely dotted] (-4-15,-5.)-- (-1-15,-5.);
			      \draw [line width=1pt, loosely dotted] (1-15,-5.)-- (4-15,-5.);
			      \draw [fill=black] (-15,0.) circle (6pt);
			      \draw [fill=black] (-4-15,-5.) circle (6pt);
			      \draw [fill=black] (4-15,-5.) circle (6pt);
			      \draw [fill=black] (-1-15,-5.) circle (6pt);
			      \draw [fill=black]  (1-15,-5) circle (6pt);
			      \draw [line width=1pt] (-15,0.)-- (-15-1,-5);
			      \draw (-1-15,-6.5) node {$j$};

			      \draw[->] (-11,0) -- (-5,0);

			      \draw [line width=1pt] (0,0)-- (0,3);
			      \draw [line width=1pt] (0,0)-- (1,-5);
			      \draw [line width=1pt] (0,0)-- (4,-5);
			      \draw [line width=1pt, loosely dotted] (1,-5)-- (4,-5);
			      \draw [fill=blue, blue] (0.,0.) circle (8pt);
			      \draw [fill=black] (1,-5) circle (6pt);
			      \draw [fill=black] (4,-5) circle (6pt);
			      \draw [color=blue] (2.5,0.5) node {\small{$n+1$}};
			      \draw  (1,-6.5) node {$k$};
			      \draw [line width=1pt, blue] (0,0)-- (-2,-5);
			      \draw [fill=black] (-2,-5) circle (6pt);
			      \draw  (-1,-5.5) node {$i$};
			      \draw [line width=1pt] (-2,-5)-- (-4,-5-5);
			      \draw [line width=1pt] (-2,-5)-- (-1,-5-5);
			      \draw [fill=black]  (-4,-5-5) circle (6pt);
			      \draw [line width=1pt, loosely dotted] (-4,-5-5)-- (-1,-5-5);
			      \draw [fill=black] (-1,-5-5) circle (6pt);
			      \draw  (-1,-5.5-6) node {$j$};
		      \end{tikzpicture}
	      \end{center}
\end{itemize}

For instance, the following sequence shows the progress of generating a labeled plane tree with seven nodes.

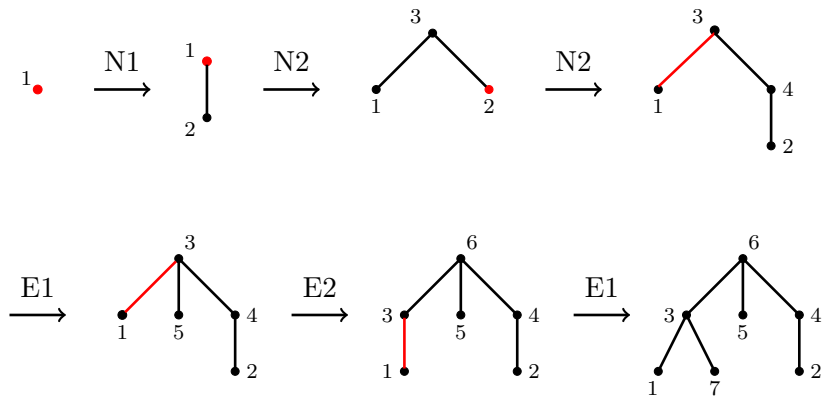
\begin{center}

	\begin{tikzpicture}[scale=0.75]
		\draw [fill=red, red]  (0.5,0) circle (2.2pt);

		\draw [->,line width=1pt] (1+0.5,0.) -- (2+0.5,0.);
		\draw[color=black] (2,0.5) node {N1};

		\draw [fill=red, red]  (3+0.5,0.5) circle (2.2pt);
		\draw [fill=black]  (3+0.5,-0.5) circle (2pt);
		\draw [line width=1pt] (3+0.5,0.44)-- (3+0.5,-0.5);

		\draw [->,line width=1pt] (4+0.5,0.) -- (5+0.5,0.);
		\draw[color=black] (5,0.5) node {N2};

		\draw [fill=black]  (7+0.5,1.) circle (2pt);
		\draw [fill=black]  (6+0.5,0) circle (2pt);
		\draw [fill=black,red]  (8+0.5,0) circle (2pt);
		\draw [line width=1pt] (7+0.5,1.)-- (6+0.5,0.);
		\draw [line width=1pt] (7+0.55,0.95)-- (8+0.45,0.05);

		\draw [->,line width=1pt] (9+0.5,0.) -- (10+0.5,0.);
		\draw[color=black] (10,0.5) node {N2};

		\draw [fill=black]  (12+0.5,1.05) circle (2.2pt);
		\draw [fill=black]  (13+0.5,0) circle (2pt);
		\draw [fill=black]  (13+0.5,-1.) circle (2pt);
		\draw [fill=black]  (11+0.5,0) circle (2pt);
		\draw [color=black, line width=1pt,red] (12+0.5,1.)-- (11+0.5,0.05);
		\draw [line width=1pt] (12+0.5,1.)-- (13+0.5,0.);
		\draw [line width=1pt] (13+0.5,0.)-- (13+0.5,-1.);

		\draw [->,line width=1pt] (0.,-4.) -- (1.,-4.);
		\draw[color=black] (0.5,-3.5) node {E1};

		\draw [fill=black]  (2,-4) circle (2.2pt);
		\draw [fill=black]  (3,-3) circle (2pt);
		\draw [fill=black]  (3,-4) circle (2pt);
		\draw [fill=black]  (4,-4) circle (2pt);
		\draw [fill=black]  (4,-5) circle (2pt);
		\draw [line width=1pt, red] (2.95,-3.05)-- (2.05,-3.95);
		\draw [line width=1pt] (3,-3)-- (3,-4);
		\draw [line width=1pt] (3,-3)-- (4,-4);
		\draw [line width=1pt] (4,-4)-- (4,-5);

		\draw [->,line width=1pt] (5,-4) -- (6,-4);
		\draw[color=black] (5.5, -3.5) node {E2};

		\draw [fill=black]  (8,-3) circle (2pt);
		\draw [fill=black]  (7,-4) circle (2pt);
		\draw [fill=black]  (7,-5) circle (2pt);
		\draw [fill=black]  (8,-4) circle (2pt);
		\draw [fill=black]  (9,-4) circle (2pt);
		\draw [fill=black]  (9,-5) circle (2pt);
		\draw [line width=1pt] (8.,-3.)-- (7.,-4.);
		\draw [line width=1pt] (8.,-3.)-- (8.,-4.);
		\draw [line width=1pt] (8.,-3.)-- (9.,-4.);
		\draw [line width=1pt] (9.,-4.)-- (9.,-5.);
		\draw [color=red, line width=1pt] (7.,-4.05)-- (7.,-5+0.05);

		\draw [->,line width=1pt] (10.,-4.) -- (11.,-4.);
		\draw[color=black] (10.5,-3.5) node {E1};
		\draw [fill=black]  (13,-3) circle (2pt);
		\draw [fill=black]  (12,-4) circle (2pt);
		\draw [fill=black]  (13,-4) circle (2pt);
		\draw [fill=black]  (11.5,-5) circle (2pt);
		\draw [fill=black]  (12.5,-5) circle (2pt);
		\draw [fill=black]  (14,-4) circle (2pt);
		\draw [fill=black]  (14,-5) circle (2pt);
		\draw [line width=1pt] (13.,-3.)-- (12.,-4.);
		\draw [line width=1pt] (13.,-3.)-- (13.,-4.);
		\draw [line width=1pt] (13.,-3.)-- (14.,-4.);
		\draw [line width=1pt] (12.,-4.)-- (11.5,-5);
		\draw [line width=1pt] (12.,-4.)-- (12.5,-5);
		\draw [line width=1pt] (14.,-4.)-- (14.,-5.);

		\begin{scriptsize}
			\draw[color=black] (0.3,0.2) node {1};
			\draw[color=black] (3.2,0.7) node {1};
			\draw[color=black](3.2,-0.7) node {2};
			\draw[color=black] (7+0.2,1.3) node {3};
			\draw[color=black] (6+0.5,-0.3) node {1};
			\draw[color=black] (8+0.5,-0.3) node {2};
			\draw[color=black] (12+0.2,1.3) node {3};
			\draw[color=black] (11+0.5,-0.3) node {1};
			\draw[color=black] (13.3+0.5,0) node {4};
			\draw[color=black] (13.3+0.5,-1) node {2};
			\draw[color=black] (3+0.2,-2.7) node {3};
			\draw[color=black] (2,-4.3) node {1};
			\draw[color=black] (3,-4.3) node {5};
			\draw[color=black] (4.3,-4) node {4};
			\draw[color=black] (4.3,-5) node {2};

			\draw[color=black] (8+0.2,-2.7) node {6};
			\draw[color=black] (6.7,-4) node {3};
			\draw[color=black] (8,-4.3) node {5};
			\draw[color=black] (9.3,-4) node {4};
			\draw[color=black] (9.3,-5) node {2};
			\draw[color=black] (6.7,-5) node {1};

			\draw[color=black] (13+0.2,-2.7) node {6};
			\draw[color=black] (11.7,-4) node {3};
			\draw[color=black] (13,-4.3) node {5};
			\draw[color=black] (14.3,-4) node {4};
			\draw[color=black] (14.3,-5) node {2};
			\draw[color=black] (11.4,-5.3) node {1};
			\draw[color=black] (12.5,-5.3) node {7};
		\end{scriptsize}
	\end{tikzpicture}
\end{center}

The trees obtained in the four manners are distinct. We claim that the insertion algorithm can generate all trees in $\mathcal{T}_{n+1}$ from trees in $\mathcal{T}_{n}$. To do this, we just need to construct the corresponding deletion algorithm for a tree $T\in \mathcal{T}_{n+1}$:
\begin{itemize}
	\item If the node $\mathbf{n+1}$ is a leaf, we delete the node directly;
	\item If the node $\mathbf{n+1}$ is an interior node with the oldest child $\mathbf{k}$, we   contract the edge $(\mathbf{n+1},\mathbf{k})$ to a node $\mathbf{k}$, and the children of $\mathbf{k}$ and $\mathbf{n+1}$ keep the previous order respectively and the children of $\mathbf{k}$ is elder than the children of $\mathbf{n+1}$.
\end{itemize}

It is not difficult to verify the deletion algorithm described above is well-defined.

As a direct application of this insertion algorithm, we can give the following combinatorial interpretation of the recurrence relation of Narayana numbers.
\begin{prop}[{\cite{Chen2022Recurrences,Ma2019gamma}}]
	The Narayana numbers satisfy the following identity.
	\begin{align}\label{rec-A-number}
		(n+2)  N(n+1, k)  =   (n+2k) {N}(n, k)+(3 n+4-2 k) {N}(n, k-1).
	\end{align}
	Therefore, the Narayana polynomials satisfy
	\begin{align}\label{rec-A-poly}
		(n+2) N_{n+1}^A(x) =  \left( (3n+2)x + n\right)N_{n}^A(x) +2(x-x^2) \left( N_{n}^A(x) \right)'.
	\end{align}
\end{prop}

\begin{proof}
	It is known that the number of labeled plane trees on $n+1$ nodes with~$k$ leaves are given by $(n+1)! N(n,k)$. We shall show that labeled plane trees on $n+2$ nodes with $k$ leaves can be obtained from labeled plane trees on $n+1$ nodes via the following two cases:
	\begin{itemize}
		\item Given a labeled plane tree on $n+1$ nodes with $k$ leaves, we have the option to perform insertion either in Case E2 for every edge in the tree, or in Case N1 and Case N2 for every leaf. Thus the number in this case is $n+2k$.
		\item Given a labeled plane tree on $n+1$ nodes with $k-1$ leaves, we have the option to perform insertion either in Case E1 for every edge in the tree, or in Case N1 and Case N2 for every non-leaf node. Thus the number in this case is $n+2(n+1-k+1)=3n+4-2k$.
	\end{itemize}
	This completes the proof of  \eqref{rec-A-number}.
\end{proof}
Note that the recurrence relation \eqref{rec-A-number}
was found by Ma, Ma, and Yeh~\cite[Lemma 9]{Ma2019gamma}, and \eqref{rec-A-poly} was found by Chen, Yang, and Zhao~\cite{Chen2022Recurrences}.

{\color{red}

}
\section{A context-free Grammar for Narayana Polynomials} \label{section-3}

\tikzstyle{every node}=[circle,inner sep=1pt,fill=white!60]
\tikzstyle{tn}=[shape=circle, draw, color=black!70]
\tikzstyle{marker}=[shape=circle,minimum size=0.2cm, draw,blue]

In this section, we first introduce the notion of \emph{improper edges} on labeled plane trees and then define a context-free grammar that can be used to generate labeled plane trees.
We use the grammar to generalize $N_n^A(x)$ and $N_n^B(x)$ to $\widetilde{N}^A_n(x,y,s,t)$ and $\widetilde{N}^B_n(x,y,s,t)$ as shown in \eqref{eq-A} and \eqref{NBSimple}, respectively.
Besides, as an application of our grammar, we give a grammatical proof of a convolution identity and a formula of the generating function of $N_n^A(x)$ and $N_n^B(x)$. For the sake of convenience, we consider a homogeneous version of $N_n^A(x)$ and $N_n^B(x)$ as
\[
	N_n^A(x,y)=\sum_{k=1}^n  \frac{1}{n}{n\choose k}{n\choose k-1} x^ky^{n-k+1},
\]
and
$$
	N_n^B(x,y)=\sum_{k=0}^n  {n\choose k}^2 x^{k}y^{n-k}.
$$
Here we set $N_0^A(x,y)=y.$



Inspired by Shor's work~\cite{Shor1995new} on Cayley trees, we define improper edges on labeled plane trees as follows.
Given a labeled plane tree $T$, we define the  $\beta$-value of a node $\mathbf{j}$, denoted $\beta(\mathbf{j})$,  as the smallest label on any node in the subtree rooted at $\mathbf{i}$ (it is possible that $\beta(\mathbf{j})={j}$).
Meanwhile, we define the $\alpha$-value of a node $\mathbf{j}$, denoted $\alpha(\mathbf{j})$, to be the minimum among the label of its father and the $\beta$-values of its elder brothers. Precisely, let $\mathbf{i}$ be the father of $\mathbf{j}$ and suppose that  $\mathbf{k_1},\mathbf{k_2},\cdots,\mathbf{k_{t-1}}, \mathbf{j},\mathbf{k_{t+1}},\cdots,\mathbf{k_m}$ are the children of $\mathbf{i}$ listed  from left to right. Then $\alpha(\mathbf{j})$ is defined as
\[
	\alpha(\mathbf{j})=\min\left( {i}, \beta(\mathbf{k_1}),\beta(\mathbf{k_2}),\cdots,\beta(\mathbf{k_{t-1}})\right).
\]
The edge $e=(\mathbf{i},\mathbf{j})$ is called \emph{proper} if $\alpha(\mathbf{j})<\beta(\mathbf{j})$; otherwise, it is called \emph{improper}.
Take the tree~$T$ in Figure~\ref{Figure_label} for instance.
The edge $(\mathbf{4},\mathbf{2})$ is an improper edge since $\alpha(\mathbf{2})=4$ and $\beta(\mathbf{2})=2$. Meanwhile, the edge $(\mathbf{6},\mathbf{4})$ is a proper edge since since $\alpha(\mathbf{4})=1$ and $\beta(\mathbf{4})=2$.

Notice that Cayley trees can be considered as labeled plane trees where the children of each node are arranged in increasing order of their $\beta$-values.
It is worth noting that an edge $e=(\mathbf{i},\mathbf{j})$ in a Cayley tree $T$ is improper if and only  $i>\beta(j)$.
Thus the definition of improper edges can be viewed as a generalization of the definition of Cayley trees given by Shor.

Let $\widetilde{N}(n,k,r)$ denote the number of labeled plane trees on $[n+1]$ with $k$ leaves and $r$ improper edges. We define a homogeneous multivariate polynomial:
\[
	\widetilde{N}_n^A(x,y,s,t)=\sum_{k=1}^{n} \sum_{r=0}^n  \widetilde{N}(n,k,r) s^{n-r} t^{r}x^{k} y^{n-k+1}.
\]
Especially, we set $\widetilde{N}_0^A(x,y,s,t)=y$. It can be seen that
\begin{align}\label{eq-A}
	\widetilde{N}_n^A(x,y,1,1)=(n+1)!N_n^A(x,y).
\end{align}


Following Chen~\cite{Chen1993Context},
a context-free grammar $G$ over an alphabet set $A=\{x, y, z, \ldots \}$ of variables is a collection of substitution rules that replace a variable in $A$ with a Laurent polynomial of variables in $A$.
The formal derivative $D$ associated with a context-free grammar $G$ over $A$ is a linear operator that acts on Laurent polynomials with variables in $A$ and  satisfies the following relations for each substitution rule:
$$\begin{aligned} & D(u+v)=D(u)+D(v), \\ & D(u v)=D(u) v+u D(v) .\end{aligned}$$
Dumont and Ramamonjisoa~\cite{Dumont1996Grammaire} found a grammar
$
	\{A\rightarrow A^3S,S\rightarrow AS^2\}
$
to generate Cayley trees.

We next pose the following grammar on the alphabet $A=\{s,t,x,y\}$,
\begin{equation}\label{grammar_simple}
	G: = \{
	s\rightarrow s(sx+ty), t\rightarrow t(sx+ty),
	x\rightarrow (s+t)xy,y\rightarrow (s+t)xy\}
\end{equation}
to generate labeled plane trees, as well as Narayana polynomials.
Let $D$ denote the formal derivative concerning the grammar \eqref{grammar_simple}.
The next theorem shows the relation between the grammar \eqref{grammar_simple} and labeled rooted trees.
\begin{thm}
	For any integer $n\geq0$, we have
	\begin{equation}\label{eq:Ma_2}
		D^n(y)=\widetilde{N}_n^A(x,y,s,t).
	\end{equation}
\end{thm}

\pf It is evident for $n=0$. For $n\geq1$, to complete the proof, we introduce a grammatical labeling of labeled rooted trees. For a tree $T\in\mathcal{T}_{n}$, we label each leaf by $x$ and label each interior node by $y$. Meanwhile, we label each proper edge by $t$ and label each improper edge by  $s$. Finally, we define the weight of $T$ to be the product of all the labels on $T$, namely,
\[
	wt(T)= s^{prop(T)}t^{imp(T)}x^{\ell(T)}y^{i(T)}.
\]
Here we use $prop(T)$,  $imp(T)$, $\ell(T)$, and $i(T)$ to denote the number of proper edges, improper edges, leaves, and interior nodes in $T$, respectively.
For instance, the tree $T$ has the grammatical labeling as shown in Figure \ref{Figure_label} and  $wt(T)=s^3 t^3 x^{4}y^3$.

\begin{figure}[h]
	\begin{center}
		\begin{tikzpicture}[scale=1]
			\node[tn,label=90:{$6(y)$}]{}[grow=down]
			[sibling distance=20mm,level distance=12mm]
			child {node [tn,label=180:{$3(y)$}](four){}[sibling distance=17mm]
					child {node [tn,label=180:{$1(x)$}](one){}
							edge from parent node[above left]{$t$}
						}
					child {node [tn,label=180:{$7(x)$}](two){}
							edge from parent node[above right ]{$s$}
						}
					edge from parent node[above left ]{$t$}
				}
			child {node [tn,label=180:{$5(x)$}](four){}
					edge from parent node[right ]{$s$}
				}
			child {node [tn,label=0:{$4(y)$}](four){}
					child {node [tn,label=0:{$2(x)$}](one){}
							edge from parent node[right ]{$t$}
						}
					edge from parent node[above right ]{$s$}
				};
		\end{tikzpicture}
		\caption{A labeled plane tree with its grammatical labeling}\label{Figure_label}
	\end{center}
\end{figure}
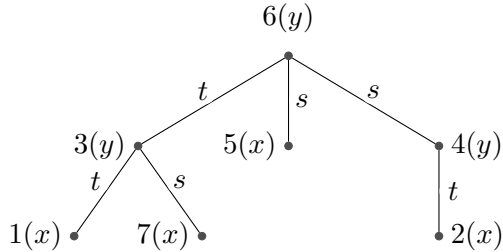
This implies that \eqref{eq:Ma_2} is equivalent to
\begin{equation}\label{eq:wt}
	D^n(y)=\sum_{T\in \mathcal{T}_{n+1}}wt(T).
\end{equation}
We prove the relation \eqref{eq:wt} by induction on $n$. For $n=1$, \eqref{eq:wt} holds since the two trees in $\mathcal{T}_2$ are with labels $sxy$ and $txy$. Now we assume that \eqref{eq:wt} holds for $n-1$, that is,
\[
	D^{n-1}(y)=\sum_{T\in \mathcal{T}_{n}}wt(T).
\]
To show that \eqref{eq:wt} holds for $n$, we turn to establish the relationship between labeled rooted trees and grammar \eqref{grammar_simple}.

Let us recall the four cases of insertion steps of the node $\mathbf{n+1}$. Here we view the insertion in Case N1 and Case N2 as an action on a node and view the insertion in Case E1 and Case E2 as an action on an edge. Consider a tree $T$ in $\mathcal{T}_{n}$.

If we add a new node $\mathbf{n+1}$ as in Case N1 and Case N2, we first choose a node $\mathbf{i}$, labeled by $x$ or $y$.
\begin{itemize}
	\item When we add $\mathbf{n+1}$ to be the old leaf of $\mathbf{i}$ as Case N1, we assign $x$ to the new leaf $\mathbf{n+1}$, and assign $s$ to the new edge $(\mathbf{i},\mathbf{n+1})$. We label  the node $\mathbf{i}$ by $y$, no matter $y$ or $x$ is assigned to  $\mathbf{i}$  in $T$.

	\item When we add $\mathbf{n+1}$ to be the father node of $\mathbf{i}$ as Case N2, we assign $y$ to the new leaf~$\mathbf{n+1}$, and assign $t$ to the new edge $(\mathbf{n+1},\mathbf{i})$. We label the node $\mathbf{i}$ by $x$, no matter $y$ or $x$ is assigned to  $\mathbf{i}$  in $T$.
\end{itemize}
Combining both two cases, whether the target node is labeled by a letter $y$ or $x$, we get $zwt+zws$, which is equivalent to the action of the substitution rule $y\rightarrow (s+t)xy$ or  $x\rightarrow (s+t)xy$.

For instance, if we apply $y\rightarrow (s+t)xy$ to the label $y$ on the node $\mathbf{6}$, we get  two new trees $T_1$ and $T_2$ in $\mathcal{T}_{8}$.
Clearly, the trees $T_1$ and $T_2$ are labeled by consistent grammatical labels as shown in Figure \ref{Figure_Addn}.

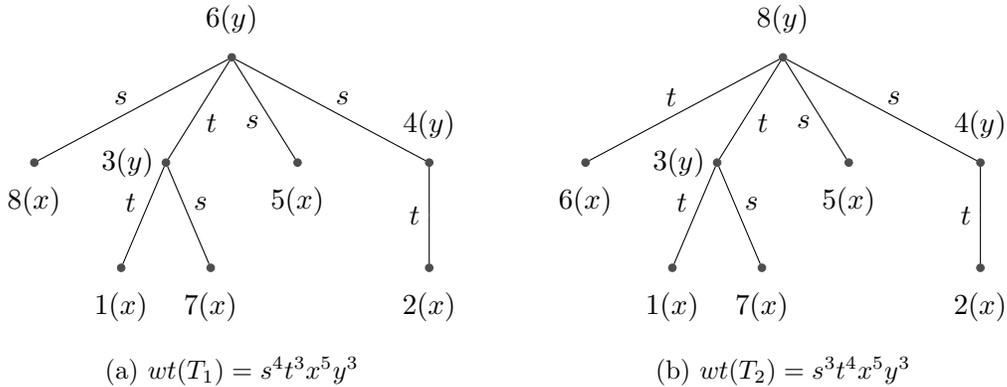
\begin{figure}[ht]
	\subcaptionbox{$wt(T_1)=s^4t^3x^5y^3$}[.5\textwidth]{%
		\begin{tikzpicture}[ scale=0.7]
			\node[tn,label=90:{$6(y)$}] {}[grow=down]
			[sibling distance=25mm,level distance=20mm]
			child {node [tn,label=270:{$8(x)$}]{}
					edge from parent node[above left=1pt ]{$s$}
				}
			child {node [tn,label=180:{$3(y)$}]{}[sibling distance=17mm]
					child {node [tn,label=270:{$1(x)$}]{}
							edge from parent node[above left=1pt ]{$t$}
						}
					child {node [tn,label=270:{$7(x)$}]{}
							edge from parent node[above right=1pt ]{$s$}
						}
					edge from parent
					node[below right=1pt ]{$t$}
				}
			child {node [tn,label=270:{$5(x)$}]{}
					edge from parent node[below left=1pt]{$s$}
				}
			child {node [tn,label=90:{$4(y)$}](four){}
					child {node [tn,label=270:{$2(x)$}]{}
							edge from parent node[left ]{$t$}
						}
					edge from parent node[above right=1pt ]{$s$}
				};
		\end{tikzpicture}
	}
	\subcaptionbox{$wt(T_2)=s^3t^4x^5y^3$}[.4\textwidth]{%
		\begin{tikzpicture}[ scale=0.7]
			\node[tn,label=90:{$8(y)$}] {}[grow=down]
			[sibling distance=25mm,level distance=20mm]
			child {node [tn,label=270:{$6(x)$}]{}
					edge from parent node[above left=1pt ]{$t$}
				}
			child {node [tn,label=180:{$3(y)$}]{}[sibling distance=17mm]
					child {node [tn,label=270:{$1(x)$}]{}
							edge from parent node[above left=1pt ]{$t$}
						}
					child {node [tn,label=270:{$7(x)$}]{}
							edge from parent node[above right=1pt ]{$s$}
						}
					edge from parent
					node[below right=1pt ]{$t$}
				}
			child {node [tn,label=270:{$5(x)$}]{}
					edge from parent node[below left=1pt]{$s$}
				}
			child {node [tn,label=90:{$4(y)$}](four){}
					child {node [tn,label=270:{$2(x)$}]{}
							edge from parent node[left ]{$t$}
						}
					edge from parent node[above right=1pt ]{$s$}
				};
		\end{tikzpicture}
	}
	\caption{An examples for $y\rightarrow (s+t)xy$}\label{Figure_Addn}
\end{figure}

If we apply $x\rightarrow (s+t)xy$ to the label $x$ on the node $\mathbf{5}$, we get  two new trees $T_3$ and $T_4$ in $\mathcal{T}_{8}$. The trees $T_3$ and $T_4$ are labeled by consistent grammatical labels as shown in Figure \ref{Figure_Addn2}.

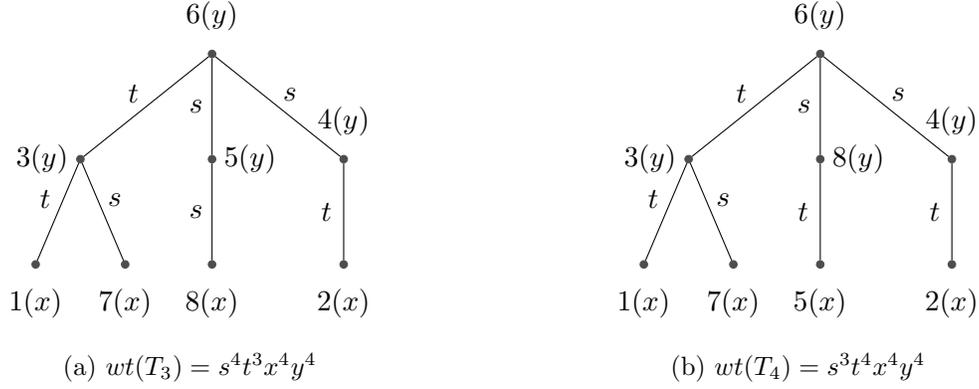
\begin{figure}[h]
	\subcaptionbox{$wt(T_3)=s^4 t^3 x^4 y^4$}[.5\textwidth]{%
		\begin{tikzpicture}[ scale=0.7]
			\node[tn,label=90:{$6(y)$}] {}[grow=down]
			[sibling distance=25mm,level distance=20mm]
			child {node [tn,label=180:{$3(y)$}]{}[sibling distance=17mm]
					child {node [tn,label=270:{$1(x)$}]{}
							edge from parent node[above left=1pt ]{$t$}
						}
					child {node [tn,label=270:{$7(x)$}]{}
							edge from parent node[above right=1pt ]{$s$}
						}
					edge from parent node[above left=1pt ]{$t$}
				}
			child {node [tn,label=0:{$5(y)$}]{}
					child {node [tn,label=270:{$8(x)$}]{}
							edge from parent node[left=1pt ]{$s$}
						}
					edge from parent node[left=1pt ]{$s$}
				}
			child {node [tn,label=90:{$4(y)$}]{}
					child {node [tn,label=270:{$2(x)$}]{}
							edge from parent node[left=1pt ]{$t$}
						}
					edge from parent node[above right=1pt ]{$s$}
				};
		\end{tikzpicture}
	}
	\subcaptionbox{$wt(T_4)=s^3 t^4 x^4 y^4$}[.5\textwidth]{%
		\begin{tikzpicture}[ scale=0.7]
			\node[tn,label=90:{$6(y)$}] {}[grow=down]
			[sibling distance=25mm,level distance=20mm]
			child {node [tn,label=180:{$3(y)$}]{}[sibling distance=17mm]
					child {node [tn,label=270:{$1(x)$}]{}
							edge from parent node[above left=1pt ]{$t$}
						}
					child {node [tn,label=270:{$7(x)$}]{}
							edge from parent node[above right=1pt ]{$s$}
						}
					edge from parent node[above left=1pt ]{$t$}
				}
			child {node [tn,label=0:{$8(y)$}]{}
					child {node [tn,label=270:{$5(x)$}]{}
							edge from parent node[left=1pt ]{$t$}
						}
					edge from parent node[left=1pt ]{$s$}
				}
			child {node [tn,label=90:{$4(y)$}]{}
					child {node [tn,label=270:{$2(x)$}]{}
							edge from parent node[left=1pt ]{$t$}
						}
					edge from parent node[above right=1pt ]{$s$}
				};
		\end{tikzpicture}
	}
	\caption{An examples for $x\rightarrow (s+t)xy$}\label{Figure_Addn2}
\end{figure}
If we add a new node $\mathbf{n+1}$ as in Case E1 and Case E2, we first choose an edge $(\mathbf{i},\mathbf{j})$, which is labeled by a letter $t$ or $s$.
\begin{itemize}
	\item When we add $\mathbf{n+1}$ to be the younger brother of $\mathbf{j}$ as Case E1, we keep all old labels on~$T$, assign $x$ to the new leaf $\mathbf{n+1}$, and assign $s$ to the new proper edge $(\mathbf{i},\mathbf{n+1})$.

	\item When we add $\mathbf{n+1}$ to be the father of $\mathbf{i}$ as Case E2, we keep all old labels on $T$, assign $y$ to the new interior node $\mathbf{n+1}$, and assign $t$ to the new edge $(\mathbf{n+1},\mathbf{i})$.
\end{itemize}
Combining both two cases,  whether the target edge is labeled by a letter $t$ or $s$, we obtain new labels $sx+ty$, which is equivalent to the action of the substitution rule $s\rightarrow s(sx+ty)$ or  $t\rightarrow t(sx+ty)$.

For instance, if we apply $s\rightarrow s(sx+ty)$ to the label $s$ on the edge $(\mathbf{3},\mathbf{7})$, we get two new trees $T_5$ and $T_6$ in $\mathcal{T}_{8}$.
Clearly, the trees $T_5$ and $T_6$ are labeled by consistent grammatical labels as shown in Figure \ref{Figure_Addn3}.

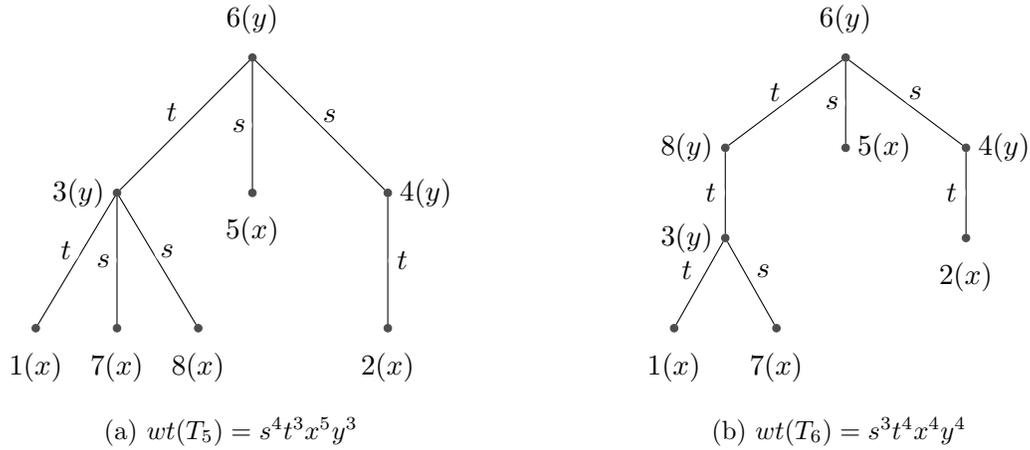
\begin{figure}[h]
	\subcaptionbox{$wt(T_5)=s^4 t^3 x^5y^3$}[.5\textwidth]{%
		\begin{tikzpicture}[ scale=0.9]
			\node[tn,label=90:{$6(y)$}] {}[grow=down]
			[sibling distance=20mm,level distance=20mm]
			child {node [tn,label=180:{$3(y)$}]{}[sibling distance=12mm]
					child {node [tn,label=270:{$1(x)$}]{}
							edge from parent node[above left ]{$t$}
						}
					child {node [tn,label=270:{$7(x)$}]{}
							edge from parent node[left ]{$s$}
						}
					child {node [tn,label=270:{$8(x)$}]{}
							edge from parent node[above right ]{$s$}
						}
					edge from parent node[above left=1pt ]{$t$}
				}
			child {node [tn,label=270:{$5(x)$}]{}
					edge from parent node[left ]{$s$}
				}
			child {node [tn,label=0:{$4(y)$}](four){}
					child {node [tn,label=270:{$2(x)$}]{}
							edge from parent node[right ]{$t$}
						}
					edge from parent node[above right ]{$s$}
				};
		\end{tikzpicture}
	}
	\subcaptionbox{$wt(T_6)=s^3 t^4 x^4 y^4$}[.5\textwidth]{%
		\begin{tikzpicture}[scale=0.8]
			\node[tn,label=90:{$6(y)$}] {}[grow=down]
			[sibling distance=20mm,level distance=15mm]

			child {node [tn,label=180:{$8(y)$}]{}[sibling distance=17mm]
					child {node [tn,label=180:{$3(y)$}]{}
							child {node [tn,label=270:{$1(x)$}]{}
									edge from parent node[above left=1pt ]{$t$}
								}
							child {node [tn,label=270:{$7(x)$}]{}
									edge from parent node[above right=1pt ]{$s$}
								}
							edge from parent node[left ]{$t$}
						}
					edge from parent node[above left ]{$t$}
				}
			child {node [tn,label=0:{$5(x)$}]{}
					edge from parent node[left ]{$s$}
				}
			child {node [tn,label=0:{$4(y)$}]{}
					child {node [tn,label=270:{$2(x)$}]{}
							edge from parent node[left ]{$t$}
						}
					edge from parent node[above right ]{$s$}
				};
		\end{tikzpicture}
	}
	\caption{An examples for $s\rightarrow s(sx+ty)$}\label{Figure_Addn3}

\end{figure}

If we apply $t\rightarrow t(sx+ty)$ to the label $t$ on the edge $(\mathbf{6},\mathbf{3})$, we get two new trees $T_7$ and $T_8$ in $\mathcal{T}_{8}$.
Clearly, the trees $T_7$ and $T_8$ are labeled by consistent grammatical labels as shown in Figure \ref{Figure_Addn4}.

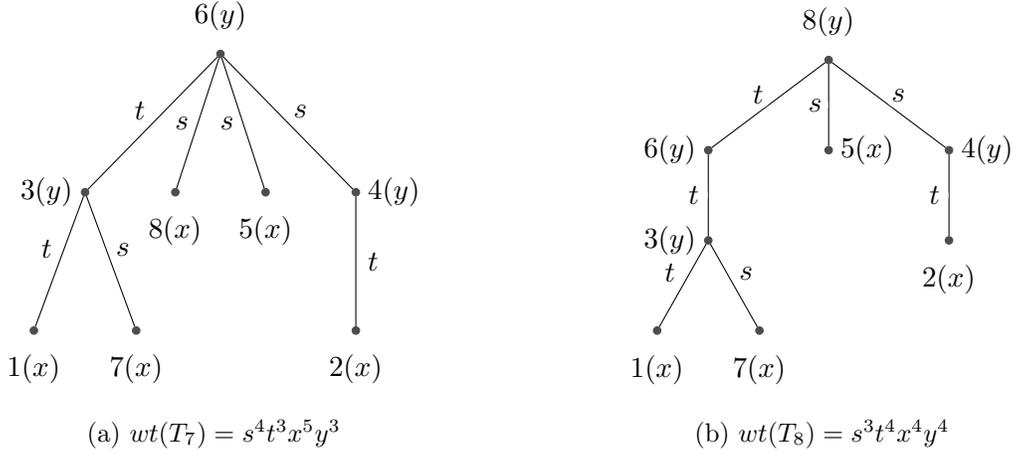
\begin{figure}[h!]
	\subcaptionbox{$wt(T_7)=s^4 t^3 x^5 y^3$}[.5\textwidth]{%
		\begin{tikzpicture}[ scale=0.8]
			\node[tn,label=90:{$6(y)$}] {}[grow=down]
			[sibling distance=15mm,level distance=23mm]
			child {node [tn,label=180:{$3(y)$}]{}[sibling distance=17mm]
					child {node [tn,label=270:{$1(x)$}]{}
							edge from parent node[above left=1pt ]{$t$}
						}
					child {node [tn,label=270:{$7(x)$}]{}
							edge from parent node[above right=1pt ]{$s$}
						}
					edge from parent node[above left=1pt ]{$t$}
				}
			child {node [tn,label=270:{$8(x)$}]{}
					edge from parent node[ left=1pt ]{$s$}
				}
			child {node [tn,label=270:{$5(x)$}]{}
					edge from parent node[ left=1pt ]{$s$}
				}
			child {node [tn,label=0:{$4(y)$}](four){}
					child {node [tn,label=270:{$2(x)$}]{}
							edge from parent node[right=1pt ]{$t$}
						}
					edge from parent node[above right=1pt ]{$s$}
				};
		\end{tikzpicture}
	}
	\subcaptionbox{$wt(T_8)=s^3 t^4 x^4 y^4$}[.5\textwidth]{%
		\begin{tikzpicture}[scale=0.8]
			\node[tn,label=90:{$8(y)$}] {}[grow=down]
			[sibling distance=20mm,level distance=15mm]

			child {node [tn,label=180:{$6(y)$}]{}[sibling distance=17mm]
					child {node [tn,label=180:{$3(y)$}]{}
							child {node [tn,label=270:{$1(x)$}]{}
									edge from parent node[above left=1pt ]{$t$}
								}
							child {node [tn,label=270:{$7(x)$}]{}
									edge from parent node[above right=1pt ]{$s$}
								}
							edge from parent node[left ]{$t$}
						}
					edge from parent node[above left ]{$t$}
				}
			child {node [tn,label=0:{$5(x)$}]{}
					edge from parent node[left ]{$s$}
				}
			child {node [tn,label=0:{$4(y)$}]{}
					child {node [tn,label=270:{$2(x)$}]{}
							edge from parent node[left ]{$t$}
						}
					edge from parent node[above right ]{$s$}
				};
		\end{tikzpicture}
	}
	\caption{An examples for $t\rightarrow t(sx+ty)$}\label{Figure_Addn4}

\end{figure}

Above all, the action of the operator $D$ on $wt(T)$ is equivalent to the insertion of $\mathbf{n+1}$ into~$T$. Thus we have
\[
	D^n(y)=D\left(\sum_{T\in \mathcal{T}_{n}}wt(T)\right)=
	\sum_{T^*\in \mathcal{T}_{n+1}}wt(T^*).
\]
This completes the proof.\qed

Next, we give a combinatorial description of Narayana polynomials of Type B in the notion of labeled rooted trees.
Let $\mathcal{T}^{*}_{n}$ denote the set of labeled plane trees on $[n+2]$ in which $\mathbf{1}$ is the old child of $\mathbf{2}$.
Let $\widetilde{N}_B(n,k,r)$ denote the number of labeled rooted trees in $\mathcal{T}^{*}_{n}$ with $k$ leaves and $r$ improper edges. Define $\widetilde{N}^B_n(x,y,s,t)$ as follows:
\[
	\widetilde{N}^B_n(x,y,s,t)=\sum_{k=0}^n  \sum_{r=0}^{n+1}\widetilde{N}_B(n,k,r)s^{n+1-r}t^r x^{k-1}y^{n-k+1}.
\]

\begin{thm}\label{thm-NB}
	For $n\geq0$, it holds that
	\begin{equation}\label{NBSimple}
		\widetilde{N}^B_n(x,y,t,t) =n!t^{n+1}N_n^B(x,y).
	\end{equation}
\end{thm}

To prove \eqref{NBSimple}, we need the following conclusion given by Chen, Deutsch, and  Elizalde~\cite{Chen2006Old}.
\begin{lem}[\cite{Chen2006Old}]\label{lem-uud}
	The number of unlabeled plane trees with $n+1$ nodes, $k$ leaves, and $i$ old leaves is
	\[
		r_{n,k,i}=\frac{1}{n}{n\choose i}{n-i\choose k-i}{n-k\choose i-1 }.
	\]
\end{lem}

\begin{proof}[Proof of Theorem~\ref{thm-NB}]

	It can be seen that
	\begin{align*}
		\widetilde{N}^B_n(x,y,t,t) & =t^{n+1} \sum_{T\in \mathcal{T}^*_n}x^{\ell(T)-1}y^{i(T)-1}                    \\
		                           & =n!\, t^{n+1} \sum_{T\in \mathcal{U}_{n+2}}ol(T)\cdot x^{\ell(T)-1}y^{i(T)-1},
	\end{align*}
	where $\mathcal{U}_{n}$ denotes the set of unlabeled rooted tress with $n$ nodes, and  $ol(T)$ denotes the number of old leaves of $T$.
	The last equation holds since a tree in $\mathcal{T}^*_n$ can be obtained from an unlabeled tree $T$ with $n+2$ nodes by labeling freely $\{3,4,\ldots,n+2\}$ on all nodes of $T$ except for an old leaf and its father.

	\allowdisplaybreaks
	It follows from Lemma \ref{lem-uud} that
	\begin{align*}
		\sum_{T\in \mathcal{U}_{n+2}}ol(T) x^{\ell(T)-1}  y^{i(T)-1} & =\sum_{k=1}^{n+1} \sum_{i=1}^{k} i \, r_{n+1,k,i}x^{k-1}y^{n-k+1}                                                  \\
		                                                             & =\sum_{k=1}^{n+1} \sum_{i=1}^{k}  \frac{i}{n+1}{n+1\choose i}{n+1-i\choose k-i}{n+1-k\choose i-1 }x^{k-1}y^{n-k+1} \\
		                                                             & =\sum_{k=1}^{n+1} \sum_{i=1}^{k}  {n\choose i-1}{n+1-i\choose k-i}{n+1-k\choose i-1 }x^{k-1}y^{n-k+1}              \\
		                                                             & =\sum_{k=1}^{n+1} \sum_{i=1}^{k}
		{n\choose k-1}{k-1\choose k-i}{n+1-k\choose i-1 }x^{k-1} y^{n-k+1}                                                                                                                \\
		                                                             & =\sum_{k=1}^{n+1}{n\choose k-1}\sum_{i=1}^k {k-1\choose k-i}{n+1-k\choose i-1 } x^{k-1}y^{n-k+1}                   \\
		                                                             & =\sum_{k=0}^{n}{n\choose k}^2x^{k}y^{n-k}.
	\end{align*}
	The second last equation holds from the famous Vandermonde's identity.
\end{proof}
\begin{thm}
	Let $D$ denote the formal derivative associated with grammar \eqref{grammar_simple}. For any integer $n\geq0$, we have     \begin{equation}\label{equ:grammar_B}
		D^n(t)=\widetilde{N}^B_n(x,y,s,t).
	\end{equation}
\end{thm}
\pf  For a tree $T$ in $\mathcal{T}^*_n$, we define the weight of $T$ as a normal labeled rooted tree, except for the labels on the node $\mathbf{1}$ and $\mathbf{2}$. Then a tree $T$ in $\mathcal{T}^*_n$ with $k$ leaves and $r$ improper edges has the weight $wt(T)=s^{n+1-r}t^r x^{k-1}y^{n-k+1}$, which implies \eqref{equ:grammar_B} is equivalent to
\begin{equation}\label{grammar_B}
	D^n(t)=\sum_{T\in\mathcal{T}^*_n}wt(T).
\end{equation}

For $n=0$, the only tree in  $\mathcal{T}^*_0$ is the tree on $[2]$ where $\mathbf{2}$ is root, whose weight is $t$. Now consider a tree in $\mathcal{T}^*_{n-1}$. Notice that there is no label in $\mathbf{1}$ and $\mathbf{2}$.
Thus we forbid the insertion of $\mathbf{n+2}$ on $\mathbf{1}$ and $\mathbf{2}$.
Since we do not insert $\mathbf{n+2}$ on $\mathbf{1}$, the node $\mathbf{1}$ is always a leaf.
Since we do not insert $\mathbf{n+2}$ on $\mathbf{2}$, the node $\mathbf{1}$ is always the first child of $\mathbf{2}$.  This completes the proof.
\qed

Setting $s=t$ in grammar \eqref{grammar_simple}, we get a grammatical description of $N_n^A(x,y)$ and $N_n^B(x,y)$.

\begin{thm}\label{Thm_TypeA}
	Let $D_H$ denote the formal derivative associated with the grammar
	\begin{equation}\label{grammar_simple-2}
		H:= \{t\rightarrow t^2(x+y),
		x\rightarrow 2txy,
		y\rightarrow 2txy\}.
	\end{equation}
	For $n\geq0$, it holds that
	\begin{align}
		D_H^n(y) & =(n+1)!t^{n}N_n^A(x,y),\label{equ:grammar-H_A} \\
		D_H^n(t) & =n!t^{n+1}N_n^B(x,y). \label{equ:grammar-H_B}
	\end{align}
	Especially,
	\begin{align}
		D_H^n(y)|_{y=1} & =(n+1)!t^nN_n^A(x),\label{equ:grammar_simple_A} \\
		D_H^n(t)|_{y=1} & =n!t^{n+1}N_n^B(x).\label{equ:grammar_simple_B}
	\end{align}
\end{thm}


This implies Equation \eqref{equ:grammar-H_B}. \qed

\begin{rem}
	Ma, Ma, and Yeh~\cite{Ma2019gamma} provided a context-free grammar
	\begin{equation}\label{g:Ma}
		\{u\rightarrow u^2v^3,v\rightarrow u^3v^2\},
	\end{equation}
	and showed that
	\begin{equation*}
		D^n(u^2)=(n+1)!\sum_{k=0}^nN(n,k)u^{3n-2k+2}v^{n+2k}
	\end{equation*}
	and
	\begin{equation*}
		D^n(uv)=n!\sum_{k=0}^n\binom{n}{k}^2 u^{3n-2k+1} v^{n+2k+1}.
	\end{equation*}
	By setting $t=uv, x=u^2, y=v^2$, the grammar \eqref{grammar_simple-2} can be changed into the grammar \eqref{g:Ma}.


\end{rem}

As an application of the grammar $H$ given in \eqref{grammar_simple-2}, we next show a grammatical proof of the following convolution identities of $N_n^A(x,y)$ and $N_n^B(x,y)$.
Note that our definition of $N_n^A(x)$ here is different from that in Petersen~\cite{Petersen2015Eulerian} due to a shift of $x$.

\begin{thm}
	For $n\geq2$, we have
	\begin{align}
		N_n^A(x,y) & = (x+y) N_{n-1}^A(x,y)+\sum_{k=2}^{n-1}N_{k-1}^A(x,y)N_{n-k}^A(x,y) \,\, (\cite[Theorem 2.2]{Petersen2015Eulerian}), \label{Conv_Nara_A} \\
		N_n^B(x,y) & = (x+y) N_{n-1}^B(x,y)+2 \sum_{k=0}^{n-2}N_{k}^B(x,y)N_{n-k-1}^A(x,y). \label{Conv_Nara_B}
	\end{align}
\end{thm}
\pf Consider the following Leibnitz-type convolution
\begin{equation}\label{Conv_Nara_grammar}
	D_H^n(t^{-2})=\sum_{k=0}^n {n\choose k}D_H^k(t^{-1})D_H^{n-k}(t^{-1}).
\end{equation}
One can verify that
\[
	D_H(t^{-2})=-2t^{-1}(x+y),\quad D_H^2(t^{-2})=2(y-x)^2,\quad D_H^n(t^{-2})=0,\mbox{ for } n\geq3,
\]
and
\[
	D_H(t^{-1})=-(x+y),\quad D_H^n(t^{-1})=-2n!t^{n-1}N_{n-1}^A(x,y),\mbox{ for } n\geq2.
\]
Thus for $n\geq 3$, the left side  of \eqref{Conv_Nara_grammar}  vanishes, while the right  side of \eqref{Conv_Nara_grammar} can be calculated as follows,
\begin{equation*}
	R.H.S.=2t^{-1}D_H^n(t^{-1})+2nD_H(t^{-1})D_H^{n-1}(t^{-1})+\sum_{k=2}^{n-2} {n\choose k}D_H^k(t^{-1})D_H^{n-k}(t^{-1}),
\end{equation*}
which can be reduced to
\[
	-4n!t^{n-2}N_{n-1}^A(x,y)+4n!(x+y)t^{n-2}N_{n-2}^A(x,y)+4n!t^{n-2}\sum_{k=2}^{n-2} N_{k-1}^A(x,y)N_{n-k-1}^A(x,y).
\]
Now  we get
\[
	N_{n-1}^A(x,y)=(x+y)N_{n-2}^A(x,y)+\sum_{k=2}^{n-2} N_{k-1}^A(x,y)N_{n-k-1}^A(x,y),
\]
which is equivalent to \eqref{Conv_Nara_A}.

Similarly, \eqref{Conv_Nara_B} can be obtained in a similar calculation from the following Leibnitz-type convolution
\begin{equation*}
	0= D_H^n(1)=\sum_{k=0}^n {n\choose k}D_H^k(t)D_H^{n-k}(t^{-1}).
\end{equation*}
This completes the proof.
\qed

Let $C^A(x,y,z)$ and $C^B(x,y,z)$  denote the generating function of Narayana polynomials of type A and type B, respectively. Namely,
\begin{align}
	C^A(x,y,z) & =\sum_{n\geq 0}N_n^A(x,y)z^n,               \\
	C^B(x,y,z) & =\sum_{n\geq 0}N_n^B(x,y)z^n.\label{Def:CB}
\end{align}
Let $\Gen(f, t)$ be the generating function of a Laurent polynomial $f$ associated with the operator $D$, as defined by
$$
	\Gen(f, t)=\sum_{n=0}^{\infty} D^n(f) \frac{t^n}{n !} .
$$
The following relations of $\Gen(f, t)$ are fisrt given by Chen \cite{Chen1993Context}:
\begin{align*}
	\Gen(f+g, t)                         & =\Gen(f, t)+\Gen(g, t),      \\
	\Gen(f\cdot g, t)                    & =\Gen(f, t)\cdot \Gen(g, t), \\
	\frac{\mbox{d}}{\mbox{dt}}\Gen(f, t) & =\Gen(D(f), t).
\end{align*}

Based on the relation \eqref{Conv_Nara_A}, Petersen \cite{Petersen2015Eulerian} provided a formula for $C^A(x,y,z)$. Here we give a grammatical proof in the following equivalent form. Our approach also leads to a formula for  $C^B(x,y,z)$.

\begin{thm}
	We have
	\begin{align}
		C^A(x,y,z) & =\frac{1+(y-x)z-\sqrt{1-2(x+y)z+(y-x)^2z^2}}{2z}\,  (\cite[Equation\,  (2.6)]{Petersen2015Eulerian}),  \label{Cwzt} \\
		C^B(x,y,z) & =\frac{1}{\sqrt{1-2(x+y)z+(y-x)^2z^2}}.\label{CBwzt}
	\end{align}
\end{thm}
\pf Notice that
\[
	D_H(t^{-2})=-2t^{-1}(y+x),\quad D_H^2(t^{-2})=2(y-x)^2,\quad D_H^n(t^{-2})=0,\mbox{ for } n\geq3,
\]
We obtain
$$\Gen(t^{-2},u)=t^{-2}-2t^{-1}(y+x)u+(y-x)^2u^2.$$
Equivalently,
\[
	\Gen(t^{-1},u)={\sqrt{t^{-2}-2t^{-1}(y+x)u+(y-x)^2u^2}}.
\]
Thus,
\[
	\Gen(t,u)=\frac{1}{\sqrt{t^{-2}-2t^{-1}(y+x)u+(y-x)^2u^2}}.
\]
It follows from \eqref{equ:grammar-H_B} and \eqref{Def:CB} that
\begin{equation*}
	\Gen(t,u)=tC^B(x,y,tu),
\end{equation*}
which leads to Equation \eqref{CBwzt}.

According to $D_H(-t^{-1})=y+x$,  it holds that
\begin{align*}
	\Gen(-t^{-1},u) & =-t^{-1}+(x+y)u+2\sum_{n\geq2}D_H^{n-1}(x)\frac{u^n}{n!} \\
	                & =-t^{-1}+(x+y)u+2\sum_{n\geq2}N^A_{n-1}(x,y)t^{n-1}u^n   \\
	                & =-t^{-1}+(x+y)u+2u \left(C^A(x,y,tu)-y \right).
\end{align*}
This implies
\[
	C^A(x,y,tu)=\frac{1+(y-x)tu-\sqrt{1-2(x+y)tu+(y-x)^2t^2u^2}}{2tu},
\]
which is equivalent to Equation \eqref{Cwzt}.
This completes the proof.
\qed

\section{Stable multivariate  Narayana polynomials}   \label{section-4}

%
%
%

The main objective of this section is to prove Theorem~\ref{thm-main}. By introducing a multivariate refinement of the Narayana grammar, we give a further generalization of $\widetilde{N}_n^A(x,y,s,t)$ and $\widetilde{N}_n^B(x,y,s,t)$, and then confirm their real stability via Borcea and Br\"and\'en's characterization of stability-preserving linear operators. Besides, two results derived from Theorem~\ref{thm-main} are given as applications of the theory of stable polynomials.

We first define a refined grammatical labeling for a labeled plane tree $T\in \mathcal{T}_{n}$.
We shall assign variables to each node and edge in $T$.
Following the previous section, we assign the variable $s$ to an edge $(\mathbf{i}, \mathbf{j})$ if it is proper, and assign the variable $t$ otherwise.
Besides, for any node $\mathbf{i}$, the label function $f(\mathbf{i})$ and the variable associated with it are defined as follows:
\begin{itemize}
	\item If $\mathbf{i}$ is a leaf node in $T$, then we let $f(\mathbf{i}) = \max \{ i, \alpha (\mathbf{i})  \}$ and assign the variable $x_{f(\mathbf{i})}$ to the leaf $\mathbf{i}$.
	\item If $\mathbf{i}$ is an interior node, we let $f(\mathbf{i})= \max\{{i, \beta(\mathbf{j})}\}$, where $\mathbf{j}$ is the old child of $\mathbf{i}$, and assign the variable $y_{f(\mathbf{i})}$ to the interior node $\mathbf{i}$.
\end{itemize}
Then, we define the refined weight of the tree $T$ on $[n]$, denoted $\mathbf{wt}(T)$, as the product of the variables corresponding to all nodes and edges of $T$. Namely,
\[
	\mathbf{wt}(T)= s^{ imp(T)} t^{n-1-imp(T)}\prod_{\mathbf{i} \in \mathcal{L}(T)} x_{f(\mathbf{i})} \prod_{\mathbf{j}  \in [n] \backslash  \mathcal{L}(T) } y_{f(\mathbf{j})},
\]
where $\mathcal{L}(T)$ is denoted the leaf set of $T$.

Take the tree $T$ in Figure \ref{Figure_label} as an example. For the leaf $\mathbf{2}$, we have $f(\mathbf{2})=4$ due to $\alpha(\mathbf{2})=4$ and thus assign the variable $x_{4}$ to the leaf $\mathbf{2}$.
Similarly, for the interior node $\mathbf{6}$, we have $f(\mathbf{6})=6$ due to $\beta(\mathbf{3})=1$ and thus assign the variable $y_{6}$ to the interior node $\mathbf{6}$.
Consequently, the tree~$T$ has the refined grammatical labeling as follows and $\mathbf{wt}(T)=s^3 t^3 x_3x_4x_5x_7y_3y_4y_6$.

\begin{figure}[h]
	\begin{center}
		\begin{tikzpicture}[scale=1]
			\node[tn,label=90:{$6(y_6)$}]{}[grow=down]
			[sibling distance=20mm,level distance=12mm]
			child {node [tn,label=180:{$3(y_3)$}](four){}[sibling distance=17mm]
					child {node [tn,label=180:{$1(x_3)$}](one){}
							edge from parent node[above left]{$t$}
						}
					child {node [tn,label=180:{$7(x_7)$}](two){}
							edge from parent node[above right ]{$s$}
						}
					edge from parent node[above left ]{$t$}
				}
			child {node [tn,label=180:{$5(x_5)$}](four){}
					edge from parent node[right ]{$s$}
				}
			child {node [tn,label=0:{$4(y_4)$}](four){}
					child {node [tn,label=0:{$2(x_4)$}](one){}
							edge from parent node[right ]{$t$}
						}
					edge from parent node[above right ]{$s$}
				};
		\end{tikzpicture}
		\caption{A labeled plane tree with its refined  grammatical labeling}\label{Figure_refined_label}
	\end{center}
\end{figure}
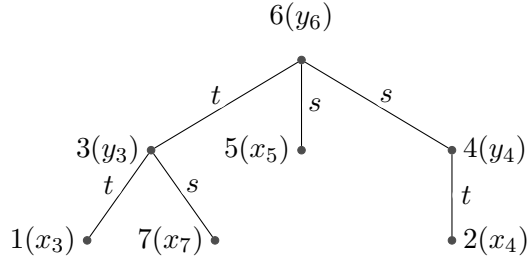

Recall that
\begin{equation}\label{Fn}
	F_n(\mathbf{x}, \mathbf{y}, s, t) = \sum_{T\in \mathcal{T}_{n+1}} \mathbf{wt}(T)
\end{equation}
and
\begin{equation}\label{Fnstar}
	F^{*}_n(\mathbf{x}, \mathbf{y}, s, t) = \sum_{T\in \mathcal{T}^*_{n}} \mathbf{wt}(T),
\end{equation}
where $\mathcal{T}_{n}$ is the set of labeled plane trees on $[n]$ and $\mathcal{T}^{*}_{n}$ is the set of labeled plane trees on $[n+2]$ in which $\mathbf{1}$ is the old leaf of $\mathbf{2}$.

We next consider the following refined context-free grammar:
$$G_n=\{
	x_k \rightarrow (s+t) x_{n+1} y_{n+1},\,
	y_k \rightarrow  (s+t) x_{n+1} y_{n+1}, \,
	s\rightarrow s(s x_{n+1}+t y_{n+1}),  \,
	t\rightarrow t(s x_{n+1}+t y_{n+1})\}.$$
From the insertion rule of the new node $\mathbf{n+1}$, one can easily derive the following relation between the grammar $G_n$ and the polynomials $F_n(\mathbf{x}, \mathbf{y}, s, t)$ and $F_n^*(\mathbf{x}, \mathbf{y}, s, t)$.

\begin{thm}
	Let $D_n$ be the linear operator associated with $G_n$. For $n\geq1$, it holds that
	\begin{equation}
		F_n(\mathbf{x}, \mathbf{y}, s, t)= D_{n} D_{n-1} \cdots D_1 (y_1)
	\end{equation}
	and
	\begin{equation}
		F^*_n(\mathbf{x}, \mathbf{y}, s, t)=  D_{n+1} D_{n}  \cdots D_2 (t).
	\end{equation}
\end{thm}

Now we prove the first part of Theorem~\ref{thm-main}.
\begin{proof}[Proof of Theorem~\ref{thm-main} (i)]
	The grammar $G_n$ can be reduced to $G$, as given in \eqref{grammar_simple}, by setting $x_k=x$ and $y_k=y$. According to Theorem \ref{Thm_TypeA}, we know that $F_n(\mathbf{x}, \mathbf{y}, s, t)$ (respectively $F^*_n(\mathbf{x}, \mathbf{y}, s, t)$ ) is a generalization of $N_n^A(x)$ (respectively $N_n^B(x)$).
\end{proof}

To prove the stability of multivariate polynomials, we need a multi-affine version of  Borcea and Br{\"a}nd{\'e}n's characterization.
A polynomial $f(\bm{x})$ is said to be \emph{multi-affine} if the power of each indeterminate $x_i$ is at most one.
For a set $\mathcal{P}$ of polynomials, let $\mathcal{P}^{MA}$ be the set of multi-affine polynomials in $\mathcal{P}$.
Borcea and Br{\"a}nd{\'e}n~\cite{Borcea2009Lee} gave a complete characterization of the linear operators which preserve stable multivariate polynomials.

\begin{lem}[{\cite[Theorem 3.5]{Wagner2011Multivariate}}]
	\label{lem-bb}
	Let $T: \mathbb{R}[\bm{x}]^{MA} \to \mathbb{R}[\bm{x}]$ be a linear operator acting on the
	variables $\bm{x}=(x_1, \dots, x_n)$.  If the polynomial
	$$G_T:=T\left(\prod_{i=1}^{n}(x_i+\hat{x}_i) \right)$$
	is a   stable polynomial  of variables $\bm{x}=(x_1,x_2,\ldots,x_n)$ and $\bm{\hat{x}}=(\hat{x}_1,\hat{x}_2,\ldots,\hat{x}_n)$,
	then $T$ preserves real stability.
\end{lem}


We next prove the second part of Theorem~\ref{thm-main}.
\begin{proof}[Proof of Theorem~\ref{thm-main} (ii)]
	Let $F_n= F_n(\mathbf{x}, \mathbf{y},s,t)$ and $F^{*}_n= F^{*}_n(\mathbf{x}, \mathbf{y},s,t)$.
	Notice that each monomial in $F_{n-1}$ is in form $s^kt^{n-1-k}f(\mathbf{x},\mathbf{y})$. Thus, $F_0=y_1$ and
	\begin{align*}
		F_n & = D_n(F_{n-1})                                                               \\
		    & =s(s x_{n+1}+t y_{n+1})\frac{\partial F_{n-1}}{\partial s}
		+t(s x_{n+1}+t y_{n+1})\frac{\partial F_{n-1}}{\partial t}                         \\
		    & \, \quad +(s+t)x_{n+1} y_{n+1} \sum_{k=1}^{n}
		\left( \frac{\partial}{\partial x_k}+\frac{\partial}{\partial y_k}  \right)F_{n-1} \\
		    & =(n-1)(s x_{n+1}+t y_{n+1})F_{n-1}+(s+t)x_{n+1} y_{n+1} \sum_{k=1}^{n}
		\left( \frac{\partial}{\partial x_k}+\frac{\partial}{\partial y_k}  \right)F_{n-1} \\
		    & =T_n(F_{n-1}),
	\end{align*}
	where $T_n$ denote the linear operator
	$$T_n=(n-1) (s\, x_{n+1}+t\, y_{n+1}) + (s+t)x_{n+1} y_{n+1} \sum_{k=1}^{n}
		\left( \frac{\partial}{\partial x_k}+\frac{\partial}{\partial y_k}  \right)$$
	in variables $\bm{x}$ and $\bm{y}$ with parameters $s$ and $t$.
	Similarly, $F^{*}_1=t(s x_3+t y_3)$ and $F^{*}_n= T_{n+1}(F^{*}_{n-1})$. It follows that the power of each indeterminate $x_i$ and $y_i$  is at most one, and thus $F_n$ and $F^{*}_n$ are multi-affine polynomials in variables $\bm{x}$ and $\bm{y}$.

	We proceed to show the linear operator $T_n$ preserves real stability. By Lemma~\ref{lem-bb}, we shall prove
	$$G_{T_n}= T_n  \left(\prod_{k=1}^{n} (x_k+\hat{x}_k) (y_k+\hat{y}_k)\right)$$
	is stable.
	The linear operator $T_n$ actions on the $\mathbf{x}$- and $\mathbf{y}$-variables and treats the variables $s, t$ as positive real constants.
	A straightforward computation leads to
	\begin{align*}
		\frac{G_{T_n}}{x_{n+1} y_{n+1}  \prod_{k=1}^{n} (x_k+\hat{x}_k) (y_k+\hat{y}_k)} &
		=
		(n-1) \left(\frac{s}{y_{n+1}}+\frac{t}{x_{n+1}}\right)  +(s+t)\sum_{k=1}^{n}
		\left( \frac{1}{x_k+\hat{x}_k}+\frac{1}{y_k+\hat{y}_k}  \right).
	\end{align*}
	Each term on the right side has a negative imaginary part whenever all variables have positive imaginary parts. Hence we prove the real stability of $G_{T_n}$. Therefore, we complete the proof of the second part of Theorem~\ref{thm-main}.
\end{proof}

Once multivariate  polynomials are shown to be real stable, we can then reduce
them to real-rooted polynomials by using the following operations.

\begin{lem}[{\cite[Lemma 2.4]{Wagner2011Multivariate}}]\label{lem-stable}
	Given $i,j \in [n]$, the following operations preserve real stability of $f \in \mathbb{R}[\bm{x}]$:
	\begin{itemize}
		\item \emph{Differentiation:} $f \mapsto \partial f/\partial x_i.$
		\item \emph{Diagonalization:} $f \mapsto f|_{x_i=x_j}.$
		\item	\emph{Specialization:} for $a \in \mathbb{R}$, $f\mapsto f|_{x_i = a}.$
	\end{itemize}
\end{lem}

As an application of Theorem~\ref{thm-main}, we derive the following result from Lemma~\ref{lem-stable}, which extends the real-rootedness of $N_n^A(x)$ and $N_n^B(x)$.
\begin{thm}
	For any positive real numbers $s$ and $t$, the polynomials
	$$\widetilde{N}_n^A(x,1,s,t) = \sum_{k=1}^n \left(\sum_{r=0}^n \widetilde{N}(n,k,r) s^{n-r} t^{r} \right) x^{k}$$  and
	$$\widetilde{N}_n^B(x,1,s,t) = \sum_{k=0}^n \left(\sum_{r=0}^{n+1} \widetilde{N}_B(n,k,r)s^{n+1-r}t^r \right) x^{k}$$ have only real roots.
\end{thm}

Another application of Theorem~\ref{thm-main} is a new version of multivariate stable refinement of the second-order Eulerian polynomial. By setting $t=0$, we restrict labeled rooted trees to those with all edges proper.
Such trees are called \emph{increasing plane trees}, the usual definition of which is requiring that the labels along any path from the root to a leaf are increasing.
Due to the famous glove bijection~\cite{Janson2008Plane}, we adopt the notion of Stirling permutations to present our result.
A permutation $\pi=\pi_1\cdots\pi_{2n}$ on the multiset $[n]_2=\{1,1,2,2,\cdots,n,n\}$ is called a \emph{Stirling permutation} if  $\pi_i=\pi_j$  implies $\pi_k<\pi_i$ for any $1\leq i<k<j\leq n$. We let $\pi_0=\pi_{2n+1}=0$ for notional convenience.
Let $\mathcal{Q}_n$ denote the set of Stirling permutations on $[n]_2$.
The second Eulerian polynomial can be interpreted as the generating polynomial of Stirling permutations over one  among ascent, plateau, and descent statistics, whose triple equidistribution was discovered by B\'ona~\cite{Bona2009Real}. For more properties about Stirling permutations, we refer the reader to~\cite{Haglund2012Stable, Janson2008Plane, Janson2011Generalized, Ma2023Stirling}.

For a Stirling permutation $\pi \in \mathcal{Q}_n$, let $P(\pi)$ denote the set of plateaux in $\pi$, namely, $P(\pi)= \{i:   \pi_{i}=\pi_{i+1}\}$.
We also let $FA(\pi)$ denote the set of ascents whose corresponding letter is the first appearance of the letter in $\pi$.  More precisely,
$$FA(\pi)= \{i: 0\leq i< 2n, \, \pi_{i}<\pi_{i+1} \mbox{ and } \pi_{j}\neq\pi_{i} \mbox{ for all } 1 \le j< i\},$$
where $\pi_0=0$.

\begin{thm}
	For any positive integer $n$,
	the multivariate polynomial of Stirling permutations
	\begin{align}\label{eq-second}
		\sum_{\pi \in \mathcal{Q}_{n}}  \, \prod_{{i} \in P(\pi)} x_{\pi_{i}} \prod_{{j} \in FA(\pi)} y_{\pi_{j+1}}
	\end{align}
	is real stable over the variables $\mathbf{x}$ and  $\mathbf{y}$.
\end{thm}

\pf Let $\mathcal{I}_n$ be the set of increasing plane trees on $[n]$.
For an increasing tree $T \in \mathcal{I}_n$, it can be seen
\[
	f(\mathbf{i})=\left\{
	\begin{aligned}
		i, & \quad \mathbf{i} \mbox{ is a leaf in }T,                                            \\
		j, & \quad  \mathbf{i} \mbox{ is an  interior node in T with its old child } \mathbf{j}.
	\end{aligned}
	\right.
\]
Thus we have
\[
	F_{n-1}(\mathbf{x},\mathbf{y},s,0)=t^{n-1}\sum_{T\in \mathcal{I}_n }\prod_{\mathbf{i} \in \mathcal{L}(T)} x_{i} \prod_{\mathbf{j}  \in [n] \backslash \mathcal{L}(T) } y_{f(\mathbf{j})}
\]
is real stable over the variables $\mathbf{x}$ and  $\mathbf{y}$.

Given an increasing tree $T \in \mathcal{I}_n$, let $\pi$ be the corresponding permutation under the glove bijection (also known as the ``depth-first walk'' )~\cite{Janson2008Plane, Janson2011Generalized}.
It is easy to verify that $\mathbf{i}$ is a leaf in $T$ if and only if the two $i$'s appearing consecutively contribute a plateau to $\pi$.
Meanwhile, $\mathbf{i}$ is an interior node with its old child $\mathbf{j}$ in $T$ if and only if the first occurrence of ${i}$ appears as an ascent bottom in $\pi$, whose next letter is ${j}$.
Removing the first and last letter 1 in $\pi$ and decreasing each
remaining letter by one, we obtain a Stirling permutation on $[n-1]_2$.
This indicates that
\[
	F_{n-1}(\mathbf{x},\mathbf{y},s,0)=t^{n-1}\sum_{\sigma \in \mathcal{Q}_{n-1}}  \, \prod_{k \in P(\sigma)} x_{\sigma_{k}+1} \prod_{\ell \in FA(\sigma)} y_{\sigma_{\ell+1}+1}.
\]
Finally, we complete the proof of the stability of the multivariate polynomial \eqref{eq-second} from that of $F_{n-1}(\mathbf{x},\mathbf{y},s,0)$.
\qed

By letting $x_{i}=x$ and $ y_{i}=1$ in~\eqref{eq-second}, our multivariate polynomial can be reduced to the second Eulerian polynomial in terms of plateau statistic.
It is remarkable to mention that our polynomial~\eqref{eq-second} is a distinct generalization of the second Eulerian polynomial from the one given by Haglund and Visontai~\cite{Haglund2012Stable}.

\vskip 3mm
\noindent {\bf Acknowledgments.}
The authors wish to thank William Y.C. Chen for his courses on context-free grammars, especially the Dumont-Ramamonjisoa grammar on Cayley trees during our graduate study at Nankai University.
The authors thank  Per Alexandersson for his helpful discussion after the first arXiv version.
Yang was supported by the National Natural Science Foundation of China (No.~12001404), and Zhang was supported by the National Natural Science Foundation of China~(No.~12171362).

\end{document}